\documentclass{article}
\usepackage{amssymb,latexsym,amsmath,amsthm,hyperref,comment}
\usepackage{xspace}  

\newtheorem{mainnm}{Theorem}
\newtheorem{meta}[mainnm]{Theorem}
\newtheorem{atleast66}[mainnm]{Theorem}
\newtheorem{allofem}[mainnm]{Theorem}

\newtheorem{pisthree}{Proposition}
\newtheorem{foursixorninestuff}[pisthree]{Proposition}
\newtheorem{twofive}[pisthree]{Proposition}
\newtheorem{oddstuff}[pisthree]{Proposition}

\newtheorem{modtwothree}{Lemma}
\newtheorem{sevensmooth}[modtwothree]{Lemma}
\newtheorem{fivesmooth}[modtwothree]{Lemma}
\newtheorem{arbsmall}[modtwothree]{Lemma}
\newtheorem{threesmooth}[modtwothree]{Lemma}
\newtheorem{meta2}[modtwothree]{Lemma}

\newtheorem{elevensmoothnmvalues}{Corollary}
\newtheorem{sevensmoothnmvalues}[elevensmoothnmvalues]{Corollary}
\newtheorem{fivesmoothnmvalues}[elevensmoothnmvalues]{Corollary}
\newtheorem{sevenoverfifteen}[elevensmoothnmvalues]{Corollary}

\newcommand{\mfree}{$M$-free\xspace}
\newcommand{\pfull}{$P$-full\xspace}
\newcommand{\free}[1]{#1-free\xspace}
\newcommand{\full}[1]{#1-full\xspace}
\newcommand{\nalpha}{n_{\alpha}\xspace}
\newcommand{\n}[1]{n_{#1}\xspace}
\newcommand{\nm}{N_M\xspace}
\newcommand{\setA}[1]{A(#1)\xspace}
\newcommand{\setB}[1]{B(#1)\xspace}

\begin{document}
\vspace*{-2cm}
\Large
 \begin{center}
Partitions with prescribed sum of reciprocals: computational results \\ 

\hspace{10pt}

\large
Wouter van Doorn \\

\hspace{10pt}

\end{center}

\hspace{10pt}

\normalsize

\vspace{-10pt}

\centerline{\bf Abstract}
For a positive rational $\alpha$, call a set of distinct positive integers $\{a_1, a_2, \ldots, a_r\}$ an $\alpha$-partition of $n$, if the sum of the $a_i$ is equal to $n$ and the sum of the reciprocals of the $a_i$ is equal to $\alpha$. Define $\nalpha$ to be the smallest positive integer such that for all $n \ge \nalpha$ an $\alpha$-partition of $n$ exists and, for a positive integer $M \ge 2$, define $\nm$ to be the smallest positive integer such that for all $n \ge \nm$ a $1$-partition of $n$ exists where $M$ does not divide any of the $a_i$. In this paper we determine $\nm$ for all $M$, and find the set of all $\alpha$ such that $\nalpha \le 100$.

\section{Introduction and overview of results}
Let us call a set of positive integers $\{a_1, \ldots, a_r\}$ an $\alpha$-partition of $n$, if all $a_i$ are distinct, $a_1 + \ldots + a_r = n$, and $\frac{1}{a_1} + \ldots + \frac{1}{a_r} = \alpha$. For a positive integer $M \ge 2$, define an $\alpha$-partition to be \mfree if none of the $a_i$ are divisible by $M$. On the other hand, for a set of primes $P$, define an $\alpha$-partition $\{a_1, \ldots, a_r\}$ to be \pfull, if all prime factors of the $a_i$ are contained in $P$. To give an example, $\{2, 3, 6\}$ is a $1$-partition of $11$, as $2 + 3 + 6 = 11$ and $\frac{1}{2} + \frac{1}{3} + \frac{1}{6} = 1$. It is furthermore \full{$\{2, 3\}$}, as $2, 3$ and $6$ only have $2$ and $3$ as prime factors. One can note that, whenever $M$ is divisible by a prime $p \notin P$, any \pfull partition is \mfree. \\

Interest in \mfree or \pfull partitions emerges quite naturally from the study of $\alpha$-partitions in the first place. For example, when Graham in \mbox{\cite{gr}} obtained the surprising result that $1$-partitions of $n$ exist for all $n \ge 78$, he actually proved that \pfull $1$-partitions of $n \ge 78$ exist, with $P$ equal to $\{2, 3, 5, 7, 11, 13\}$. And in the same paper he then immediately uses this stronger result to prove the existence of $1$-partitions of large $n$ with arbitrarily large smallest element. \\

Where \mfree and \pfull partitions really shine however, is in the study of $\nalpha$; the smallest positive integer such that there exists an $\alpha$-partition of $n$, for all $n \ge \nalpha$. For example, it turns out that \mfree $\alpha$-partitions for just a single $\alpha$ (which will often be $\alpha = 1$ in our case), can help pin down the value of $\n{\alpha'}$, for infinitely many other values of $\alpha'$. This paper is therefore dedicated to proving the existence of such \mfree $1$-partitions, and in a follow-up paper we will then apply these results to bound quantities like $\nalpha$ and $\left|\{\alpha : \nalpha \le n \}\right|$. \\ 

And on the topic of \mfree $1$-partitions, in \mbox{\cite{ks2}} K\"ohler and Spilker improve on Graham's original result, by showing that \full{$\{2, 3, 5\}$} $1$-partitions of $n$ exist for all $n \ge 113$. In particular this implies that \mfree $1$-partitions exist for all $n \ge 113$ and all $M$ divisible by a prime larger than $5$. They then asked whether this can be further improved to the existence of \full{$\{2, 3\}$} $1$-partitions for all large enough $n$. In Section \ref{fiveparts} we will answer this question in the affirmative, by showing that \full{$\{2, 3\}$} $1$-partitions exist for all $n \ge 531$. \\

This leaves open the question whether \mfree $1$-partitions exist for large enough $n$, when $M$ can be written as $M = 2^x 3^y$. In Section \ref{foursixornineparts} we will then deal with the case $x + y \ge 2$, and show that \mfree $1$-partitions exist for all $n \ge 183$, if $M$ is divisible by $4, 6$ or $9$. So the only two remaining cases are $M = 2$ and $M = 3$. \\

For these latter two cases we come to a negative result, as we will show that a \free{$2$} $\alpha$-partition of $n$ implies $n \equiv \alpha \pmod{8}$, while a \free{$3$} $\alpha$-partition of $n$ implies $n \equiv \alpha \pmod{3}$. On the other hand, for $n \equiv 1 \pmod{3}$ we prove in Section \ref{threeparts} that \full{$\{2, 5\}$} $1$-partitions exist if $n \ge 3634$. And in Section \ref{oddparts} we show for $n \equiv 1 \pmod{8}$ that \full{$\{3, 5, 7\}$} $1$-partitions exist if $n \ge 3609$, answering a question raised by K\"ohler and Spilker in \mbox{\cite{ks}}. Since all $M \ge 2$ are then dealt with, this concludes our results on \mfree $1$-partitions. \\

Let $\setA{n}$ be defined as the set of all positive rationals $\alpha$ with $\nalpha \le n$. After showing $1 \in \setA{78}$ in \mbox{\cite{gr}}, Graham proved more generally that the infinite union $\bigcup_{n \ge 1} \setA{n}$ is equal to the set of all positive rationals. To date, this is (to the best of our knowledge) still all that is known about these sets $\setA{n}$. In Section \ref{nalpha} we will ameliorate this state of affairs by fully determining $\setA{n}$ for all $n \le 100$. In particular, we will see that $\setA{65}$ is the empty set, while $\setA{100}$ turns out to contain $4314$ rationals, with the set difference $\setA{n} \setminus \setA{n-1}$ non-empty for all $n$ with $66 \le n \le 100$. \\

Perhaps the most interesting result we obtain along the way is a metatheorem we prove in Section \ref{metastuff}. For a given set $S$ of positive rationals and any property $Q$ that an $\alpha$-partition might have, this metatheorem provides a general framework for how a proof can be structured, if one wants to show that for all large enough $n$ and all $\alpha \in S$, an $\alpha$-partition of $n$ with property $Q$ exists. \\

With Theorem $6$ in \mbox{\cite{squares}}, Alekseyev proved a special case of such a metatheorem, in order to show that for all positive integers $n > 15707$ a set $\{a_1, \ldots, a_r\}$ of positive integers exists with $\frac{1}{a_1} + \ldots + \frac{1}{a_r} = 1$, $a_1^2 + \ldots + a_r^2 = n$, and $a_i \ge 6$ for all $i$. \\

Finally, a few words on the notation we use. If $A = \{a_1, \ldots, a_r\}$ is a set of positive integers and $m \in \mathbb{N}$, then $\sum A$ denotes $a_1 + \ldots + a_r$, $\sum A^{-1}$ is defined as the sum $\frac{1}{a_1} + \ldots + \frac{1}{a_r}$, and we write $mA$ for the set $\{ma_1, \ldots, ma_r\}$.

\section[Nm values]{Values of $\nm$ for all $M \ge 2$}
Define $\nm$ to be the smallest positive integer such that for all integers $n \ge \nm$ an \mfree $1$-partition of $n$ exists. As it turns out, for $M = 2$ and $M = 3$ such an $\nm$ does not exist, due to the following lemma.

\begin{modtwothree} \label{modtwothree}
Let $A$ be a set of integers. If all elements of $A$ are odd, then $\sum A \equiv \sum A^{-1} \pmod{8}$. Similarly, if none of the integers in $A$ are divisible by $3$, then $\sum A \equiv \sum A^{-1} \pmod{3}$.
\end{modtwothree}

\begin{proof}
With $M'$ equal to either $8$ or $3$, this follows from the fact that every integer coprime to $M'$ is its own inverse modulo $M'$. That is, $a \equiv \frac{1}{a} \pmod{M'}$ for all integers $a$ coprime to $M'$, and hence $\sum A \equiv \sum A^{-1} \pmod{M'}$. 
\end{proof}

In particular, if $\sum A^{-1} = 1$, then $\sum A \equiv 1 \pmod{M'}$. For $M = 2$ and $M = 3$ we will therefore redefine $\nm$ to the smallest positive integer such that an \mfree $1$-partition of $n$ exists for $n = \nm$, and for all integers $n \ge \nm$ with $n \equiv 1 \pmod{8}$ or $n \equiv 1 \pmod{3}$ respectively. We are then able to state the full theorem on the values of $\nm$, that we will then prove in the coming sections.

\begin{mainnm} \label{main1}
The values of $\nm$ can be found in the following table. For all positive integers $M \ge 2$ not present in this table we have $\nm = 78$.

\begin{table}[h]
\centering
\def\arraystretch{1.3}
\begin{tabular}{|l|l|l|l|l|l|l|l|l|l|l|l|}
\hline
$M$ & $2$         & $3$   & $4$   & $5$   & $6$   & $7$  & $8$   & $9$  & $10$  & $11$  & $12$  \\ \hline
$\nm$ & $\le 737$ & $154$ & $155$ & $126$ & $183$ & $97$ & $101$ & $91$ & $108$ & $92$  & $98$ \\ \hline \hline
$M$ & $14$        & $15$  & $16$  & $18$  & $20$  & $21$ & $22$  & $24$ & $28$  & $30$  & $33$ \\ \hline
$\nm$ & $81$      & $108$ & $78$  & $91$  & $106$ & $81$ & $92$  & $80$ & $81$  & $108$ & $92$ \\ \hline
\end{tabular}
\end{table}
\end{mainnm}

We will have to split up the proof of Theorem \ref{main1} into a few different cases, depending on the existence of certain divisors of $M$. For example, in \mbox{\cite{ks2}} K\"ohler and Spilker prove that \full{$\{2, 3, 5\}$} $1$-partitions exist for all $n \ge 113$. This in principle already takes care of all $M$ with a prime divisor larger than $5$, and in Section \ref{largestuff} we will do just that. In the sections afterwards we will then deal with the case where the largest prime factor of $M$ is at most $5$. As for $M = 2$, it is possible to combine a brute-force search with arguments similar to ones used in \cite{shiu} (which, for example, imply that every \mfree $1$-partition of $n \le 729$ can only use integers whose largest prime factor is either $47$ or at most $31$) to prove that $\nm$ is exactly equal to $737$, but this becomes hard to verify. We therefore only claim the upper bound $\nm \le 737$ (as opposed to equality) for $M = 2$.

\section[Large divisors]{When $M$ has a prime divisor larger than $5$} \label{largestuff}
In \mbox{\cite{ks2}} K\"ohler and Spilker prove that $91$ is the only positive integer $n \ge 78$ for which a \full{$\{2, 3, 5, 7\}$} $1$-partition does not exist. This is due to the fact that $\{3, 4, 6, 11, 12, 22, 33\}$ is the unique $1$-partition of $91$. This immediately gives the following corollary.

\begin{elevensmoothnmvalues} \label{elevensmoothnm}
If $M \in \{11, 22, 33\}$, then $\nm = 92$. And if the largest prime divisor of $M$ is at least $11$ with $M \notin \{11, 22, 33\}$, then $\nm = 78$.
\end{elevensmoothnmvalues}

When the largest prime divisor of $M$ is (exactly) equal to $7$, we apply the following lemma.

\begin{sevensmooth} \label{sevensmooth}
If the largest prime divisor of $M$ is equal to $7$, then either $M \le 28$, or $M$ is divisible by $35, 42, 49, 56$ or $63$.
\end{sevensmooth}

\begin{proof}
If $\frac{M}{7}$ is divisible by $5$ or $7$, then $M$ is divisible by $35$ or $49$. On the other hand, if $\frac{M}{7}$ is not divisible by $5$ or $7$, then the largest prime divisor of $\frac{M}{7}$ is at most $3$. Now there are three different cases to consider: either $\frac{M}{7}$ is not divisible by $2$, or $\frac{M}{7}$ is not divisible by $3$, or $\frac{M}{7}$ is divisible by both $2$ and $3$. In the first case, either $M \le 21$ or $63$ divides $M$. In the second case, either $M \le 28$ or $56$ divides $M$. And in the third case, $M$ is divisible by $42$.
\end{proof}

\begin{sevensmoothnmvalues} \label{sevensmoothnm}
If $M = 7$, then $\nm = 97$. If $M \in \{14, 21, 28\}$, then $\nm = 81$. And if the largest prime divisor of $M$ is equal to $7$ with $M > 28$, then $\nm = 78$.
\end{sevensmoothnmvalues}

\begin{proof}
On the author's GitHub page\footnote{See https://github.com/Woett/Partition-data.} one can find all $1$-partions of $n$ for all $n \le 200$. In particular, if $M$ is divisible by $35, 42, 49, 56$ or $63$, then \mfree $1$-partitions of $n$ exist if $78 \le n \le 112$. As for $M \in \{14, 21, 28\}$, we note that $\{2, 4, 10, 15, 21, 28\}$ is the unique $1$-partition of $80$, while \mfree $1$-partitions of $n$ exist if $81 \le n \le 112$. And for $M = 7$, there are a total of four $1$-partitions of $96$, but all of them contain an integer divisible by $M$, while \mfree $1$-partitions of $n$ exist if $97 \le n \le 112$. Combining all this with the proof by K\"ohler and Spilker \mbox{\cite{ks2}} that \full{$\{2, 3, 5\}$} $1$-partitions exist for all $n \ge 113$, finishes the proof.
\end{proof}

\section{A metatheorem} \label{metastuff}
Since most of the proofs in the rest of this paper have essentially the same structure, we will first prove a kind of metatheorem that shows how specific proofs fit into a more general framework. \\

Let $S$ be a set of positive rationals such that for all $\alpha \in S$ we want to prove that for all large enough $n$, $\alpha$-partitions of $n$ exist with a certain property $Q$. To prove this, we choose for every $\alpha \in S$ a positive integer $l$, positive integers $m_1, \ldots, m_l$, rationals $\beta_1, \ldots, \beta_{l}$ and sets of positive integers $A_1, \ldots, A_{l}$ (with $A_i$ non-empty if $m_i = 1$), with the following five properties:

\begin{enumerate}
	\item For all $i$, $\beta_i \in S$.
	\item For every positive integer $n$, there exists an $i$ with $n \equiv \sum A_i \pmod{m_i}$.
	\item For all $i$, $\alpha = \sum A_i^{-1} + \frac{\beta_i}{m_i}$.
	\item For all $i$ and every $\beta_i$-partition $B$ with property $Q$, we have $A_i \cap m_i B = \emptyset$.
	\item For all $i$ and every $\beta_i$-partition $B$ with property $Q$, the union $A_i \cup m_i B$ also has property $Q$.
\end{enumerate}

\begin{meta} \label{meta}
Assume that for all $\alpha \in S$ such $m_i, \beta_i, A_i$ exist, satisfying the above five properties. Further assume that a positive integer $X$ exists such that for all $\alpha \in S$ and all $n$ with $X \le n \le \max_{i, \alpha} \big(\sum A_i + m_i (X-1)\big)$ an $\alpha$-partition of $n$ exists with property $Q$. Then $\alpha$-partitions of $n$ with property $Q$ exist for all $\alpha \in S$ and all $n \ge X$.
\end{meta}

\begin{proof}
By induction we may assume that $n > \max_{i, \alpha} \big(\sum A_i + m_i (X-1)\big)$ is a positive integer such that for all $\alpha \in S$ and all $n'$ with $X \le n' < n$ an $\alpha$-partition of $n'$ with property $Q$ exists. The goal is to prove that for all $\alpha \in S$ an $\alpha$-partition of $n$ with property $Q$ exists as well. So let $\alpha \in S$ be given, and let $i$ be such that $n \equiv \sum A_i \pmod{m_i}$. By writing $n = \sum A_i + m_i n'$, we claim $n' \ge X$. This latter inequality follows, because if $n' \le X-1$, then $n \le \sum A_i + m_i (X-1)$, contrary to the assumption $n > \max_{i, \alpha} \big(\sum A_i + m_i (X-1)\big)$. We may therefore apply the induction hypothesis to $n'$, so let $B$ be a $\beta_i$-partition of $n'$ with property $Q$. Then we claim that the union $A := A_i \cup m_i B$ is an $\alpha$-partition of $n$ with property $Q$. \\

Indeed, $A_i$ and $m_i B$ are disjoint by the fourth property above, while $A$ has property $Q$ by the fifth property. We get that the sum $\sum A^{-1}$ of reciprocals is equal to $\sum A_i^{-1} + \sum (m_i B)^{-1} = \sum A_i^{-1} + \frac{\beta_i}{m_i} = \alpha$ by the third property. And finally, the sum $\sum A$ is equal to $\sum A_i + \sum m_iB = \sum A_i + m_i n' = n$.
\end{proof}

To give an example, let us say $S = \{1 \}$ and let $Q$ be the property of not containing $1$ or $39$. We can then define $l = 2$, $m_1 = m_2 = 2$, $\beta_1 = 1$, $\beta_2 = 1$, $A_1 = \{3, 7, 78, 91 \}$ and $A_2 = \{2 \}$. The reader can then verify that the five properties listed above all hold, where for the fourth property one should realize that $2B$ only contains even integers, but -due to property $Q$- does not contain $2$ or $78$. \\

By Theorem \ref{meta} it is then sufficient to find a positive integer $X$ and $1$-partitions of $n$ with property $Q$ for all $n$ with $X \le n \le \max_{i} \big(\sum A_i + m_i (X-1)\big) = 2X + 177$, and it turns out that $X = 78$ works. This is how one can recreate Graham's original proof \mbox{\cite{gr}} that every positive integer $n \ge 78$ has a $1$-partition with property $Q$, as a specific instance of our metatheorem. \\

With $m = m_1 = m_2$, one could have also written the above proof more succinctly as the following table, which in principle contains all necessary data, besides the verification for all $n$ with $X \le n \le 2X + 177$. 

\begin{table}[ht]
\centering
\def\arraystretch{1.3} 
\begin{tabular}{|c|c|c|c|c|c|}
\hline
$\alpha$   & $m$  & $\beta_1$  & $\beta_2$  & $A_1$                & $A_2$     \\ \hline \hline
$1$        & $2$  & $1$        & $1$        & $\{3, 7, 78, 91 \}$  & $\{2 \}$  \\ \hline
\end{tabular}
\end{table}

For most of our proofs we will actually have $l = m_1 = \ldots = m_l$, in which case we will denote the integer that is equal to all of these by $m$. Moreover, the sets $A_i$ will then often be such that $m$ does not divide any of the elements in $A_i$, which makes verifying the fourth property trivial. \\

Another thing we would like to note in regards to Theorem \ref{meta} is that the easiest type of $\alpha$ to deal with are generally those with $\alpha < 2$ and $\alpha - 1 \in S$. For those $\alpha$, one can (unless adding $1$ to a $\beta_1$-partition ruins property $Q$) simply choose $m = 1$, $\beta_1 = \alpha - 1$ and $A_1 = \{1 \}$. And as long as $X > 1$, this works for $\alpha = 2$ as well. In fact, adding $\alpha + 1$ to $S$ for some $\alpha \le 1$ is one example of many where it can be helpful to increase the size of $S$. Even though increasing $S$ means we need to find more $m_i, \beta_i, A_i$, it also means we have more $\beta_i$ to work with, in order to satisfy the third property. Going back to Graham's proof, we could have also used $S = \{1, \frac{4}{3}, 2\}$, together with the following proof table.

\begin{table}[ht]
\centering
\def\arraystretch{1.3} 
\begin{tabular}{|c|c|c|c|c|c|}
\hline
$\alpha$ & $m$           & $\beta_1$  & $\beta_2$  & $A_1$     & $A_2$               \\ \hline \hline
$1$      & $2$           & $4/3$      & $2$        & $\{3 \}$  & $\emptyset$         \\ \hline
$4/3$    & $2$           & $2$        & $4/3$      & $\{3 \}$  & $\{3, 5, 9, 45 \}$  \\ \hline
$2$      & $1$           & $1$        & -          & $\{1\}$   & -                   \\ \hline
\end{tabular}
\end{table}

With this data it can be checked that $X = 79$ works. And not only do we now prove something about $\alpha = \frac{4}{3}$ and $\alpha = 2$ as well, $\max_{i, \alpha} \sum A_i$ has decreased from $179$ to $62$ in the process. The latter means that we have to verify fewer partitions to get the induction proof started. \\

And on the topic of verifying partitions, usually with an induction proof you start out with one or more base cases and then do the induction step. However, for a proof in the style of Theorem \ref{meta} it often make sense to do the induction step at the start, as for the base cases you have to check all $n$ with $X \le n \le \max_{i, \alpha} \big(\sum A_i + m_i (X-1)\big)$. But for this latter quantity to make sense, you first have to figure out what the values of $m_i, \beta_i$ and $A_i$ are going to be. \\

In any case, for our research the computer has been an invaluable tool, both for finding $m_i, \beta_i, A_i$ for every $\alpha \in S$, and especially for verifying that appropriate $\alpha$-partitions exist for all $n$ with $X \le n \le \max_{i, \alpha} \big(\sum A_i + m_i (X-1)\big)$, for some $X$. Hence the title of this paper.

\section[Small divisors]{When $M$ is divisible by $5$} \label{fiveparts}
In \mbox{\cite{ks2}} K\"ohler and Spilker ask whether every large enough integer $n$ has a $1$-partition $A$ such that the largest prime factor of $a$ is at most $3$ for all $a \in A$. This turns out to be true.

\begin{pisthree} \label{pthree}
If $M$ is divisible by $5$, then \mfree $1$-partitions exist for all $n \ge 126$. Moreover, \full{$\{2, 3\}$} $1$-partitions exist for all $n \ge 531$.
\end{pisthree}

\begin{proof}
For an odd prime $p$, define the sets $S_p := \{\frac{4}{p^2}, \frac{6}{p^2}, \ldots, \frac{2p^2 - 2p}{p^2}, 1, 2\}$ and $P_p := \{2, p\}$. Our proof will then fit in the framework of Theorem \ref{meta}, with $p = 3$, $S = S_p$ and the property $Q$ the property of being \full{$P_p$}. For the rest of this proof we will write $p$ instead of $3$, in order to come back to this proof later on and find out what needs to be changed in order to make it work for $p > 3$ as well. \\

In any case, as we mentioned in Section \ref{metastuff}, for $\alpha = 2$, we can choose $m = 1$, $\beta_1 = 1$ and $A_1 = \{1 \}$, while for all $\alpha \neq 2$, we choose $m = 2$. With $\alpha = \frac{s}{p^2}$ (where $4 \le s \le 2p^2 - 2p$), the rationals $\beta_1$ and $\beta_2$, and the sets $A_1$ and $A_2$, will then depend on the value of $s$. \\

For example, for $s \le p^2-p$ we choose $\beta_1 = \frac{2s-2}{p^2}$ and $\beta_2 = \frac{2s}{p^2}$, which we claim work with $A_1 = \{p^2 \}$ and $A_2 = \emptyset$. Let us check the five listed properties to verify this. First of all, $\beta_1, \beta_2 \in S_p$, since $6 \le 2s - 2 < 2s \le 2p^2 - 2p$. As for the second property, $\sum A_1 = p^2 \equiv 1 \pmod{2}$, while $\sum A_2 = 0 \equiv 0 \pmod{2}$. Thirdly, $\alpha = \frac{s}{p^2}$ is indeed equal to $\sum A_1^{-1} + \frac{\beta_1}{2} = \frac{1}{p^2} + \frac{2s-2}{2p^2} = \frac{2s}{2p^2} = \sum A_2^{-1} + \frac{\beta_2}{2}$. And finally, the fourth and fifth property follow because all sets $A_i$ that we create in this proof will only contain powers of $p$. \\

As for $\alpha = \frac{s}{p^2}$ with $p^2 - p < s \le 2p^2 - 2p$, there are three different cases we need to consider. For all these three types of $\alpha$ we once again need to choose $\beta_1$, $\beta_2$, $A_1$ and $A_2$ in such a way that the five properties hold. And analogously to what we saw in Section \ref{metastuff}, we can succinctly write this data down in a table. The motivated reader can then verify that in all three cases the five properties do indeed hold.

\begin{table}[ht]
\centering
\def\arraystretch{1.3}
\begin{tabular}{|c|c|c|c|c|}
\hline
$\alpha = sp^{-2}$             & $\beta_1p^2$  & $\beta_2p^2$ & $A_1$       & $A_2$          \\ \hline \hline
$p^2 - p < s \le p^2$          & $2s-2p$       & $2s-2p-2$    & $\{p \}$    & $\{p, p^2 \}$  \\ \hline
$s = p^2 + 1$                  & $2p^{-2}$     & $2p^2-2p$    & $\{p^2 \}$  & $\{p, p^2 \}$  \\ \hline
$p^2 + 3 \le s \le 2p^2 - 2p$  & $2s-2p^2$     & $2s-2p^2-2$  & $\{1 \}$    & $\{1, p^2 \}$  \\ \hline 
\end{tabular}
\end{table}

With $X = 814$, on the author's GitHub page one can find \free{$5$} $1$-partitions of $n$ for all $n$ with $126 \le n \le 530$, \full{$P_3$} $1$-partitions of $n$ for all $n$ with $531 \le n < X$, and \full{$P_3$} $\alpha$-partitions of $n$ for all $n$ with $X \le n \le \max_{i, \alpha} \big(\sum A_i + m (X-1)\big) = 2X + 10 = 1638$ and all $\alpha \in S_3 \setminus \{2 \}$. By Theorem \ref{meta}, this finishes the proof of Proposition \ref{pthree}.
\end{proof}

Proposition \ref{pthree} takes care of all $M$ which have $5$ as their largest prime factor, via the following lemma and subsequent corollary.

\begin{fivesmooth} \label{fivesmooth}
If the largest prime divisor of $M$ is equal to $5$, then either $M \in \{5, 10, 15, 20, 30 \}$, or $M$ is divisible by $25, 40, 45$ or $60$.
\end{fivesmooth}

\begin{proof}
One can prove this in analogy with the proof of Lemma \ref{sevensmooth}. 
\end{proof}

\begin{fivesmoothnmvalues} \label{fivesmoothnm}
If $M = 5$, then $\nm = 126$. If $M \in \{10, 15, 30\}$, then $\nm = 108$, while $\nm = 106$ if $M = 20$. And if the largest prime divisor of $M$ is equal to $5$ with $M \notin \{5, 10, 15, 20, 30\}$, then $\nm = 78$.
\end{fivesmoothnmvalues}

\begin{proof}
We leave this one for the interested reader, and it can of course be done analogously to the proof of Corollary \ref{sevensmoothnm}.
\end{proof}

In a follow-up paper we will need a result on \free{$5$} partitions for some other rationals as well, so let us quickly state and prove it here.

\begin{arbsmall} \label{arbsmall}
For every integer $k \ge 2$ there exists an $N(k) \le 106 \cdot 4^k - 98 \cdot 3^k$ such that for every $\alpha \in \{\frac{2}{3^{k-1}}, \frac{4}{3^k} \}$ and every $n \ge N(k)$ a \free{$5$} $\alpha$-partition of $n$ exists.
\end{arbsmall}

\begin{proof}
For $k = 2$ this follows from the data and the proof of Proposition \ref{pthree}, so assume by induction $k \ge 3$ and $N(k-1) \le 106 \cdot 4^{k-1} - 98 \cdot 3^{k-1}$. For $\alpha = \frac{2}{3^{k-1}}$, define $l = 3$, $m_1 = m_3 = 4$, $m_2 = 2$, $\beta_1 = \beta_2 = \beta_3 = \frac{4}{3^{k-1}}$, and define the sets $A_i$ as follows:
\begin{align*}
A_1 &:= \{3^{k-1} \} \\
A_2 &:= \emptyset \\
A_3 &:= \{2 \cdot 3^{k-1}, 3^{k}, 2 \cdot 3^{k} \}
\end{align*}

We then get $\sum A_1 \equiv \sum A_3 \equiv 1 \pmod{2}$, while $\sum A_1 \not \equiv \sum A_3 \pmod{4}$, so that for all $n \in \mathbb{N}$ there exists an $i$ with $n \equiv \sum A_i \pmod{m_i}$. \\

Let $n \ge 4N(k-1) + 11 \cdot 3^{k-1}$ be a positive integer, let $i$ be such that $n \equiv \sum A_i \pmod{m_i}$, and write $n = \sum A_i + m_i n'$ with $n' \ge N(k-1)$. With $B$ a \free{$5$} $\beta_i$-partition of $n'$, one can then check that $A_i \cup m_iB$ is a \free{$5$} $\alpha$-partition of $n$. \\

As for $\alpha = \frac{4}{3^k}$, define $m = 4$, $\beta_1 = \beta_3 = \beta_4 = \frac{4}{3^{k-1}}$, $\beta_2 = \frac{2}{3^{k-1}}$, and define the sets $A_i$ as follows:
\begin{align*}
A_1 &:= \{3^k \} \\
A_2 &:= \{2 \cdot 3^{k-1}, 2 \cdot 3^{k}, 3^{k+1}, 3^{k+2}, 2 \cdot 3^{k+2} \} \\
A_3 &:= \{2 \cdot 3^k, 3^{k+1}, 2 \cdot 3^{k+1} \} \\
A_4 &:= \{2 \cdot 3^k, 3^{k+1}, 3^{k+2}, 2 \cdot 3^{k+2} \}
\end{align*}

Here we have $\sum A_1 \equiv \sum A_3 \equiv 1 \pmod{2}$ and $\sum A_2 \equiv \sum A_4 \equiv 0 \pmod{2}$, while on the other hand, $\sum A_1 \not \equiv \sum A_3 \pmod{4}$, and $\sum A_2 \not \equiv \sum A_4 \pmod{4}$. \\

The rest of the proof for $n \ge 4N(k-1) + 98 \cdot 3^{k-1}$ follows analogously to the proof for $\alpha = \frac{2}{3^{k-1}}$, and we conclude that $N(k)$ exists with $N(k) \le 4N(k-1) + 98 \cdot 3^{k-1} \le 106 \cdot 4^{k} - 98 \cdot 3^{k}$.
\end{proof}

\section[When M is divisible by 4, 6 or 9]{When $M$ is divisible by $4$, $6$ or $9$} \label{foursixornineparts}
The only $M$ which are not covered by the results in \mbox{\cite{ks2}} or Proposition \ref{pthree}, are those which can be written as $M = 2^x 3^y$ for non-negative integers $x, y$. In this section we will assume $x + y \ge 2$, in which case \mfree $1$-partitions exist for all $n \ge 183$.

\begin{foursixorninestuff} \label{foursixorninestuff}
If $4$ divides $M$, then \mfree $1$-partitions exist for all $n \ge 155$. If $6$ divides $M$, then \mfree $1$-partitions exist for all $n \ge 183$. And if $9$ divides $M$, then \mfree $1$-partitions exist for all $n \ge 91$. 
\end{foursixorninestuff}

\begin{proof}
Let $M$ be divisible by $4$, $6$ or $9$. In the framework of Theorem \ref{meta} with property $Q$ the property of being \mfree, we choose $S = \{\frac{5}{6}, 1\}$ with $m = 5$ and $\beta_i = \frac{5}{6}$ for all $\alpha \in S$ and all $i$. For $\alpha = 1$ the sets $A_1, \ldots, A_5$ are defined as follows: 
\begin{align*}
A_{1} &:= \{2, 11, 13, 21, 22, 26, 33, 273 \} \\	
A_{2} &:= \{2, 7, 11, 21, 22, 154 \} \\ 					
A_{3} &:= \{2, 7, 13, 14, 39, 91, 182 \} \\ 			
A_{4} &:= \{2, 7, 14, 21, 23, 46, 161 \} \\ 			
A_{5} &:= \{2, 3 \} 															
\end{align*}

For $\alpha = \frac{5}{6}$ the sets $A_1, \ldots, A_5$ are defined as follows:
\begin{align*}
A_{1} &:= \{3, 7, 13, 14, 26, 273 \} \\ 					
A_{2} &:= \{2, 14, 22, 33, 77, 154 \} \\ 					
A_{3} &:= \{2, 11, 22, 33 \} \\ 									
A_{4} &:= \{3, 7, 13, 14, 39, 91, 182 \} \\ 			
A_{5} &:= \{2, 11, 26, 33, 143 \} 								
\end{align*}

For all $\alpha$ and $i$ we now have $\sum A_i \equiv i \pmod{5}$ and $\sum A_i^{-1} = \alpha - \frac{1}{6}$, thereby satisfying the second and third property. Since none of the $A_i$ contain any multiples of $4$, $5$, $6$ or $9$, one sees that the fourth and fifth properties are also satisfied. \\

With $X = 211$ and $Y \in \{155, 183, 91\}$ depending on whether $M$ is divisible by $4, 6$ or $9$, on the author's GitHub page one can find \mfree $1$-partitions of $n$ for all $n$ with $Y \le n \le \max_{i, \alpha} \big(\sum A_i + m (X-1)\big) = 5X + 396 = 1451$ and \mfree $\frac{5}{6}$-partitions of $n$ for all $n$ with $X \le n \le 5X + 396$, finishing the proof.
\end{proof}

The $\nm$ values in Theorem \ref{main1} for all $M > 3$ that are not divisible by a prime larger than $3$ now follow from the following lemma, to be proven in analogy with the proof of Lemma \ref{sevensmooth}.

\begin{threesmooth} \label{threesmooth}
If the largest prime divisor of $M > 3$ is at most $3$, then either $M \in \{4, 6, 8, 9, 12, 18, 24 \}$, or $M$ is divisible by $16, 27$ or $36$.
\end{threesmooth}


This settles the $\nm$ values for all $M > 3$.

\section[When M equals 3]{When $M$ equals $3$} \label{threeparts}
Even though the proof of the induction step in Proposition \ref{pthree} does still work when $p$ is larger than $3$, the base case of Theorem \ref{meta} cannot be true, due to Lemma \ref{modtwothree}; if $M$ equals $2$ or $3$, then not all $n$ with $X \le n \le \max_{i, \alpha} \big(\sum A_i + m_i (X-1)\big)$ can have an \mfree $\alpha$-partition. With $M'$ equal to $8$ or $3$ (depending on whether $M = 2$ or $M = 3$ respectively), at best only those $n$ with $n \equiv \alpha \pmod{M'}$ can have an \mfree $\alpha$-partition. Fortunately, we can modify Theorem \ref{meta} as follows.

\begin{meta2}[Modification of Theorem \ref{meta}] \label{meta2}
Assume that for all $\alpha \in S$, $m_i, \beta_i, A_i$ exist which satisfy the five properties mentioned at the start of Section \ref{metastuff}, and such that $\gcd(m_i, M') = \gcd(a, M') = 1$, for all $\alpha$, $i$ and $a \in A_i$. Further assume that a positive integer $X$ exists such that for all $\alpha \in S$ and all $n \equiv \alpha \pmod{M'}$ with $X \le n \le \max_{i, \alpha} \big(\sum A_i + m_i (X-1)\big)$ an \mfree $\alpha$-partition of $n$ exists with property $Q$. Then \mfree $\alpha$-partitions of $n \equiv \alpha \pmod{M'}$ with property $Q$ exist for all $\alpha \in S$ and all $n \ge X$.
\end{meta2}

\begin{proof}
Recall that, for a given $\alpha \in S$ and $i$ such that $n \equiv \sum A_i \pmod{m_i}$, in the proof of Theorem \ref{meta} we write $n = \sum A_i + m_i n'$. We then apply the induction hypothesis to $n'$ and assume a $\beta_i$-partition of $n'$. But with the amended base case, this assumption would only be valid if $n' \equiv \beta_i \pmod{M'}$. Fortunately, this turns out to be true due to the second property, if $n \equiv \alpha \pmod{M'}$.
\begin{align*}
n' &= m_i^{-1}\left(n - \sum A_i\right) \\
&\equiv m_i \left(\alpha - \sum A_i^{-1} \right) \pmod{M'} \\
&= \beta_i
\end{align*}

The rest of the proof of Theorem \ref{meta} still holds verbatim.
\end{proof}

We conclude that the framework of Theorem \ref{meta} does still suit our needs, even for \mfree partitions with $M \in \{2, 3 \}$, if we use the amended base case and as long as all $m_i$ (or all $m$ if the $m_i$ are the same for all $i$) and $a \in A_i$ are coprime to $M'$ for all $\alpha$ and $i$. Now, the proof of Proposition \ref{pthree} uses sets $A_i$ that only contain powers of $p$, while $m \in \{1, 2\}$ for all $\alpha \in S_p$. So we indeed get $\gcd(m, M') = \gcd(a, M') = 1$, and it is therefore sufficient to verify the base case.

\begin{twofive} \label{twofive}
Assume $n \equiv 1 \pmod{3}$. Then \free{$3$} $1$-partitions exist if $n \ge 154$, while \full{$\{2, 5\}$} $1$-partitions exist if $n \ge 3634$. 
\end{twofive}

\begin{proof}
With $S_p$ and $P_p$ as in the proof of Proposition \ref{pthree} and $X = 6482$, on the author's GitHub page one can find \free{$3$} $1$-partitions of $n$ for all $n \equiv 1 \pmod{3}$ with $154 \le n < 3634$, \full{$P_5$} $1$-partitions of $n$ for all $n \equiv 1 \pmod{3}$ with $3634 \le n < X$, and \full{$P_5$} $\alpha$-partitions of $n$ for all $n \equiv \alpha \pmod{3}$ with $X \le n \le \max_{i, \alpha} \big(\sum A_i + m (X-1)\big) = 2X + 28 = 12992$ and all $\alpha \in S_5 \setminus \{2 \}$. 
\end{proof}

On the other hand, all $34$ $1$-partitions of $151$ contain an integer divisible by $3$, and we get that $\nm$ is exactly equal to $154$ for $M = 3$.

\section[When M equals 2]{When $M$ equals $2$} \label{oddparts}
In analogy with Proposition \ref{twofive}, we can prove that \free{$2$} $1$-partitions do exist for all large enough integers $n$ with $n \equiv 1 \pmod{8}$.

\begin{oddstuff} \label{oddstuff}
Assume $n \equiv 1 \pmod{8}$. Then \free{$2$} $1$-partitions exist if $n \ge 737$, while \full{$\{3, 5, 7\}$} $1$-partitions exist if $n \ge 3609$. 
\end{oddstuff}

\begin{proof}
With $\alpha = 1$, we choose $S = \{\alpha \}$, $m = 15$ (which is coprime to $8$), $\beta_i = \alpha$ for all $i$, and we define the sets $A_1, \ldots, A_{15}$ as follows:
\begin{align*}
A_{1} &:= \{3, 5, 7, 9, 21, 27, 35, 81, 147, 189, 245, 441, 567, 3969 \} \\ 
A_{2} &:= \{3, 5, 7, 9, 15, 25, 63, 81, 189, 441, 567, 1225, 1323, 3969 \} \\ 
A_{3} &:= \{3, 5, 7, 9, 15, 25, 49, 125, 147, 441, 1225, 1715, 3087, 6125 \} \\ 
A_{4} &:= \{3, 5, 7, 9, 21, 27, 35, 63, 147, 189, 245, 1323 \} \\ 
A_{5} &:= \{3, 5, 7, 9, 21, 25, 27, 125, 189, 245, 441, 1225, 1323, 6125 \} \\ 
A_{6} &:= \{3, 5, 7, 9, 21, 25, 35, 63, 147, 245, 441, 1225 \} \\ 
A_{7} &:= \{3, 5, 7, 9, 15, 27, 49, 81, 189, 441, 567, 3969 \} \\ 
A_{8} &:= \{3, 5, 7, 9, 21, 25, 35, 63, 175, 245, 441, 1029, 1225, 8575 \} \\ 
A_{9} &:= \{3, 5, 7, 9, 15, 25, 35, 147, 441, 1225, 1715, 3087 \} \\ 
A_{10} &:= \{3, 5, 7, 9, 15, 27, 49, 63, 189, 1323 \} \\ 
A_{11} &:= \{3, 5, 7, 9, 15, 27, 49, 63, 343, 441, 1323, 9261 \} \\ 
A_{12} &:= \{3, 5, 7, 9, 15, 25, 49, 63, 441, 1225 \} \\ 
A_{13} &:= \{3, 5, 7, 9, 15, 21, 35, 441, 1715, 3087 \} \\ 
A_{14} &:= \{3, 5, 7, 9, 21, 27, 35, 49, 189, 245, 441, 1323 \} \\ 
A_{15} &:= \{3, 5, 7, 9, 25, 27, 35, 63, 81, 189, 245, 567, 1225, 3969 \} 
\end{align*}

One can now check $\sum A_i \equiv i \pmod{15}$ and $\sum A_i^{-1} = \frac{14}{15}$ for all $i$, taking care of the second and third property. With property $Q$ defined as the property of being \full{$\{3, 5, 7\}$} and not containing $1$, the fourth and fifth property are also clear, since all sets $A_i$ consist of integers only divisible by $3, 5$ or $7$, and do not contain any multiples of $15$ larger than $15$ itself. \\

With $X = 3609$, on the author's GitHub page one can find \free{$2$} $1$-partitions of $n$ for all $n \equiv 1 \pmod{8}$ with $737 \le n < X$, and $1$-partitions of $n$ with property $Q$ for all $n \equiv 1 \pmod{8}$ with $X \le n \le \max_{i, \alpha} \big(\sum A_i + m (X-1)\big) = 15X + 12963 = 67098$, which finishes the proof.
\end{proof}

For future reference we note that all \free{$2$} $1$-partitions $A$ of $n \ge 737$ we create have $1 \notin A$ and $\{3, 5\} \subset A$. Considering the set $C = A \setminus \{3, 5\}$ then provides the following corollary.

\begin{sevenoverfifteen} \label{sevenfifteen}
If $n \ge 729$ and $n \equiv 1 \pmod{8}$, then a \free{$2$} $\frac{7}{15}$-partition $C$ of $n$ exists, with $C \cap \{1, 3, 5\} = \emptyset$.
\end{sevenoverfifteen}

\section[Finding rationals with bound at most 100]{The set of all $\alpha$ with $\nalpha \le 100$} \label{nalpha}
Define $\setB{n}$ to be the set of all rationals $\alpha$ for which an $\alpha$-partition of $n$ exists, and let us define $\setB{n,N} := \displaystyle \bigcap_{i = n}^{N} \setB{i}$. Since $\setA{n}$ is equal to the infinite intersection, we get in particular $\setA{n} \subseteq \setB{n,N}$ for all $N \ge n$. The following pseudo-code shows how to create these sets $\setB{n}$ and $\setB{n,N}$, where in practice the function \verb|forpart| from the computer program PARI/GP can be applied to loop over all partitions of a positive integer. \\

\textbf{Algorithm: Create\_\setB{n}}

\begin{tabbing}
  \hspace{2em} \= \hspace{2em} \= \hspace{2em} \= \kill
  1. Initialize \texttt{B(n) = $\{ \}$}. \\
  2. For \texttt{partition p of n}: \\
  \> If \texttt{p only contains distinct integers}: \\
  \> \> a. \texttt{$\alpha$ = sum of reciprocals of p}. \\
	\> \> b. \texttt{B(n) = B(n) $\cup$} \texttt{$\{ \alpha \}$}. \\
  3. Return \texttt{B(n)}.
\end{tabbing}

\textbf{Algorithm: Create\_\setB{n,N}}

\begin{tabbing}
  \hspace{2em} \= \hspace{2em} \= \hspace{2em} \= \hspace{2em} \= \kill
  1. Initialize \texttt{B(n,n) = Create\_B(n)}. \\
	2. For \texttt{i going from n+1 to N}: \\
	\> Initialize \texttt{B(n,i) = $\{ \}$}. \\
  \> For \texttt{partition p of i}: \\
  \> \> If \texttt{p only contains distinct integers}: \\
	\> \> \> a. \texttt{$\alpha$ = sum of reciprocals of p}. \\ 
	\> \> \> b. If \texttt{$\alpha \in$ \hspace{-7pt} B(n,i-1)}: \\
	\> \> \> \> \texttt{B(n,i) = B(n,i) $\cup$} \texttt{$\{ \alpha \}$}.  \\
  3. Return \texttt{B(n,N)}.
\end{tabbing}

Graham proved that for every positive rational $\alpha$ there exists an $\nalpha$ such that $\alpha \in \setA{\nalpha}$. On the other hand, a uniform lower bound exists.

\begin{atleast66}
For all positive rationals $\alpha$ we have $\nalpha \ge 66$.
\end{atleast66}

\begin{proof}
Writing and running the code for \emph{Create\_\setB{n,N}} with $n = 65$ and $N = 78$ gives $\setB{n,N} = \emptyset$. In particular, $\setA{65} \subseteq \setB{65, 78} = \emptyset$. So $\alpha \notin \setA{m} \subseteq \setA{65} $ for all $\alpha$ and all $m \le 65$, so that $\nalpha \ge 66$.
\end{proof}

\begin{allofem}
There exist exactly $4314$ positive rationals $\alpha$ with $\nalpha \le 100$.
\end{allofem}

\begin{proof}
Writing and running the code for \emph{Create\_\setB{n,N}} with $n = 100$ and $N = 136$ gives $|\setB{n,N}| = 4314$, which provides an upper bound on the number of $\alpha$ with $\nalpha \le 100$. To prove that for all $\alpha \in \setB{n,N}$ we actually have $\alpha \in \setA{n}$, we once again apply the framework of Theorem \ref{meta}. \\

For all $\alpha \in \setB{n,N}$ with $\alpha \le 2$ and $\alpha - 1 \in \setB{n,N}$ (and there are $772$ of these), we choose $m = 1$, $\beta_1 = \alpha - 1$ and $A_1 = \{1 \}$. For all other $\alpha$ we choose $m = 2$. This does mean that, in total, we need to choose more than $7000$ rationals $\beta_i$ and more than $7000$ sets $A_i$. As the margins of this paper are unfortunately too narrow to contain all of this, these $\beta_1, \beta_2, A_1, A_2$ can all be found on the author's GitHub page, along with $\alpha$-partitions of $n'$ for all $\alpha \in \setB{n,N}$ and all $n'$ with $X = 100 \le n' \le \max_{i, \alpha} \big(\sum A_i + m (X-1)\big) = 240$. 
\end{proof}

As a small compensation and to at least show something, in the following table one can find the values of $m, \beta_i, A_i$ for all nine rationals $\alpha$ with $\nalpha \le 70$. We should point out however that this does not comprise a self-contained proof for those $\alpha$ with $\nalpha \le 70$, as it uses rationals $\beta_i \in \setA{100} \setminus \setA{70}$, violating the first property of Theorem \ref{meta}.

\begin{table}[ht]
\centering
\def\arraystretch{1.3}
\begin{tabular}{|c|c|c|c|c|c|}
\hline
$\alpha$ & $m$   & $\beta_1$ & $\beta_2$  & $A_1$     & $A_2$        \\ \hline \hline
$4/5$    & $2$   & $14/15$   & $8/5$      & $\{3 \}$  & $\emptyset$  \\ \hline
$7/12$   & $2$   & $1/2$     & $7/6$      & $\{3 \}$  & $\emptyset$  \\ \hline
$9/5$    & $1$   & $4/5$     & -          & $\{1 \}$  & -            \\ \hline
$11/12$  & $2$   & $7/6$     & $11/6$     & $\{3 \}$  & $\emptyset$  \\ \hline
$13/12$  & $2$   & $3/2$     & $13/6$     & $\{3 \}$  & $\emptyset$  \\ \hline
$19/12$  & $1$   & $7/12$    & -          & $\{1 \}$  & -            \\ \hline
$23/12$  & $1$   & $11/12$   & -          & $\{1 \}$  & -            \\ \hline
$25/12$  & $2$   & $13/6$    & $3/2$      & $\{1 \}$  & $\{1, 3 \}$  \\ \hline
$97/60$  & $1$   & $37/60$   & -          & $\{1 \}$  & -            \\ \hline
\end{tabular}
\end{table}

With the data we have generated, we can also look at the growth rate of $|\setA{n}|$ as a function of $n$. Down below, we have tabulated the number of $\alpha$ for which $\nalpha = n$, for all $n$ with $65 \le n \le 100$. Equivalently, this is the cardinality of the set $\setA{n} \setminus \setA{n-1}$. 

\begin{table}[ht]
\centering
\def\arraystretch{1.3}
\begin{tabular}{|c|l|l|l|l|l|l|l|l|l|l|l|l|l|}
\hline
$n$                             & $65$  & $66$  & $67$  & $68$  & $69$  & $70$  & $71$  & $72$  & $73$  & $74$  & $75$  & $76$  \\ \hline
$|\{\alpha : \nalpha = n\}|$    & $0$   & $2$   & $2$   & $2$   & $1$   & $2$   & $4$   & $5$   & $5$   & $7$   & $7$   & $5$   \\ \hline \hline
$n$                             & $77$  & $78$  & $79$  & $80$  & $81$  & $82$  & $83$  & $84$  & $85$  & $86$  & $87$  & $88$  \\ \hline
$|\{\alpha : \nalpha = n\}|$    & $12$  & $18$  & $22$  & $32$  & $38$  & $41$  & $48$  & $57$  & $76$  & $82$  & $74$  & $97$  \\ \hline \hline
$n$                             & $89$  & $90$  & $91$  & $92$  & $93$  & $94$  & $95$  & $96$  & $97$  & $98$  & $99$  & $100$ \\ \hline
$|\{\alpha : \nalpha = n\}|$    & $117$ & $155$ & $170$ & $194$ & $228$ & $277$ & $306$ & $332$ & $430$ & $473$ & $483$ & $510$ \\ \hline
\end{tabular}
\end{table}
 
The two elements in $\setA{66}$ are $\frac{4}{5}$ and $\frac{11}{12}$, which in turn implies that the two elements in $\setA{67} \setminus \setA{66}$ are $\frac{9}{5}$ and $\frac{23}{12}$. And of course, $\alpha = 1$ shows up as one of the eighteen elements with $\nalpha = 78$. \\

As it turns out, it is possible to use the existence of \mfree $\alpha$-partitions in order to non-trivially relate the growth rates of $|\setA{n}|$ and $|\setB{n}|$ to one another, and show the following inequalities:

\begin{equation*}
e^{\left(\frac{\pi}{\sqrt{3}} - o(1)\right)\sqrt{\frac{n}{\log(n)}}} < |\setA{n}| \le |\setB{n}| < e^{\pi\sqrt{\frac{n}{3}}}
\end{equation*}

As the proofs of these statements are more theoretical in nature and less reliant on computational data, we shall leave this for a follow-up paper.

\end{document}